\documentclass[12pt,letterpaper,reqno]{amsart}

\usepackage[letterpaper, margin=1 in]{geometry}

\newif\iffigures\figurestrue
\figuresfalse

\usepackage{amssymb,amsmath,mathrsfs}
\usepackage{bm}
\usepackage{bbm}
\usepackage{amsmath}
\usepackage{amsthm}
\usepackage{amsfonts}
\usepackage{graphicx}
\usepackage{amsmath}
\usepackage{amsthm}
\usepackage{amssymb}
\usepackage{verbatim}
\usepackage{color}
\usepackage[longnamesfirst]{natbib}
\bibpunct[ ]{(}{)}{,}{a}{}{,}

\defcitealias{KVV}{KVV}

\newcommand{\norm}[1]{\left\Vert#1\right\Vert}
\newcommand{\abs}[1]{\left\vert#1\right\vert}
\newcommand{\set}[1]{\left\{#1\right\}}
\newcommand{\Comp}{\mathbb C\,}
\newcommand{\Real}{\mathbb R}

\newcommand{\ev}{\mathbb E}

\newcommand{\X}{{Q}}

\newcommand{\sch}{{\sf Sch}_\tau}
\newcommand{\schs}{{\sf Sch}^*_\tau}
\newcommand{\zb}{\overline z}

\newcommand{\psil}{\psi^{\la}}

\newcommand{\ql}{q^{\la}}

\newcommand{\cA}{\mathcal{A}}
\newcommand{\cB}{\mathcal{B}}

\newcommand{\cE}{\mathcal{E}}
\newcommand{\cF}{\mathcal{F}}

\newcommand{\cP}{\mathcal{P}}

\newcommand{\cM}{\mathcal{M}}

\newcommand{\cW}{\mathcal{W}}

\newcommand{\cZ}{\mathcal{Z}}

\newcommand{\cR}{\mathcal{R}}

\newcommand{\bin}[2]{\left (
\begin{array} {c}
#1 \\[-3pt]
#2
\end{array}
\right )}

\newcommand{\supp}{{\rm{supp}}}

\newcommand\vect[2]{
\left(\!\begin{array}{c}
#1\\
#2
\end{array}\!\right)}

\newcommand{\mat}[4]{\left( \begin{array}{cc}
#1 & #2  \\
#3 & #4  \\
\end{array} \right)}

\newcommand{\nfloor}{\lfloor n t \rfloor}

\makeatletter
\newcommand{\xRightarrow}[2][]{\ext@arrow 0359\Rightarrowfill@{#1}{#2}}
\makeatother

\newcommand{\nfloortau}{\lfloor nt/\tau \rfloor}

\newcommand{\lrb}[1]{\left( #1 \right) }

\def\text#1{\hbox{#1}}
\def\Z{{\mathbb Z}}
\def\R{{\mathbb R}}
\def\N{{\mathbb N}}

\def\C{{\mathbb C}}
\def\P{\mathbf P}
\def\E{\mathbf E}

\def\eps{\epsilon}

\newcommand{\of}[1]{\left ( #1 \right ) }

\newcommand{\im}{\mathrm{Im}}
\newcommand{\re}{\mathrm{Re}}

\newcommand{\la}{\lambda}

\newcommand{\To}{\rightarrow}

\newcommand{\vectv}[2]
{
   \begin{pmatrix} #1 \\ #2 \end{pmatrix}
}
\newcommand{\vecth}[2]
{
   \begin{pmatrix} #1 & #2 \end{pmatrix}
}

\newcommand{\eo}{\vectv{1}{0}}
\newcommand{ \eot}{\vecth{1}{0}}

\newtheorem{theorem}{Theorem}[section]
\newtheorem{thm}{Theorem}[section]

\newtheorem{lemma}[theorem]{Lemma}
\newtheorem{corollary}[theorem]{Corollary}
\newtheorem{cor}[theorem]{Corollary}

\newtheorem{conjecture}[theorem]{Conjecture}

\newtheorem{remark}[theorem]{Remark}

\newtheorem{prop}[theorem]{Proposition}

\theoremstyle{definition}
\numberwithin{equation}{section}

\begin{document}

\title{Eigenvectors of the 1-dimensional critical random Schr\"odinger operator}
\author{Ben Rifkind \and B\'alint Vir\'ag}

\maketitle

\begin{abstract}
The purpose of this paper is to understand in more detail the shape of the eigenvectors of the  random Schr\"odinger operator $H = \Delta  + V$ on $\ell^2(\Z)$. Here $\Delta$ is the discrete Laplacian and $V$ is a random potential.  It is well known that under certain assumptions on $V$  the spectrum of this operator is  pure point and its
eigenvectors are exponentially localized; a phenomenon known as Anderson Localization.  We restrict  the operator to $\Z_n$ and consider the critical  model,
$$
(H_n\psi)_\ell =\psi_{\ell -1,n}+\psi_{\ell +1,n}+v_{\ell,n} \psi_\ell , \quad \psi_0=\psi_{n+1}=0,
$$
with $ v_k$ are 
independent random variables with mean $0$ and variance
$\sigma^2/n$.

\bigskip

\noindent We show that the scaling limit of the shape of a uniformly chosen eigenvector of $H_n$ is
$$
\exp \left ( - \frac{|t-U|}{4} + \frac{\cZ_{t-U}}{\sqrt{2}}  \right ),
$$
where $U$ is uniform on $[0,1]$ and $\cZ$ is an independent two sided Brownian motion started from $0$.
\end{abstract}

\section{Introduction}

We consider the critical model of one-dimensional discrete random
Schr\"{o}dinger  operators given by the matrix
\begin{equation}\label{shrod1dmatrix}
H_n=\left( \begin{array}{cccccc}
v_{1,n} & 1 &  &  &  & \\
1 & v_{2,n} & 1 &  & & \\
  & 1  &\ddots &\ddots & &\\
& & \ddots & \ddots &1 & \\
& & & 1 & v_{n-1,n} & 1 \\
& & & & 1 & v_{n,n} \\
\end{array} \right)
\end{equation}
where
\begin{equation}
v_{k,n}=\sigma \omega_k/\sqrt{n}.
\end{equation}
Here $\omega_k$ are independent random variables with mean
$0$, variance $1$ and bounded third absolute moment.

When $\sigma = 0$, $H_n$ is a perturbation of the discrete Laplacian of the one-dimensional box $\Z_n$. The eigenvalues are periodic and the eigenvectors are not localized. The eigenvalues $\mu_k$ and eigenvectors $ \psi_k$ of $H_n$ are given by
\begin{align*}
\mu_{k} &= 2\cos(\pi k/(n+1)) \\
\psi_{k}(\ell) &= \sin(\pi k \ell/(n+1)) \quad \ell=1,\dots,n.
\end{align*}

If $\sigma>0$ and the variance of $v_{k,n}$ does not depend on $n$, the eigenvectors are localized (\cite{C87}, \cite{kunz80}, \cite{GMP}) and the local statistics of eigenvalues are Poisson (\cite{minami1996}, \cite{molcanov1980}), even together with the localization centers, \cite{Killip2006}.

The critical regime where the variance $v_{k,n}$ scales like $n^{-1}$ was introduced in \citet{KVV} (cited here as KVV) in order to investigate the transition between localized and extended states.  For very small $\sigma$, $H_n$ is a perturbation of the discrete Laplacian. The eigenvectors are extended and the eigenvalues are close to periodic. While for very large $\sigma$, $H_n$ behaves like a diagonal matrix. The eigenvalues are independent (Poisson statistics) and the eigenvectors are localized.  It was proven in \citetalias{KVV} that the eigenvectors of $H_n$ are delocalized and that the transfer matrix evolution has a scaling limit. Building on the framework developed in that paper, here we focus further on scaling limits of the eigenvectors of $H_n$ and give a simple description of their limiting shape. We now explain what we mean by the shape of an eigenvector.

Under the scaling of the critical model the model $H_n$ is a perturbation of the non-noise case and as the eigenvectors are delocalized and highly oscillatory there can be no functional scaling limit of the eigenvectors.
\begin{figure}[ht!]
\centering
\includegraphics[width=160mm, height=60mm]{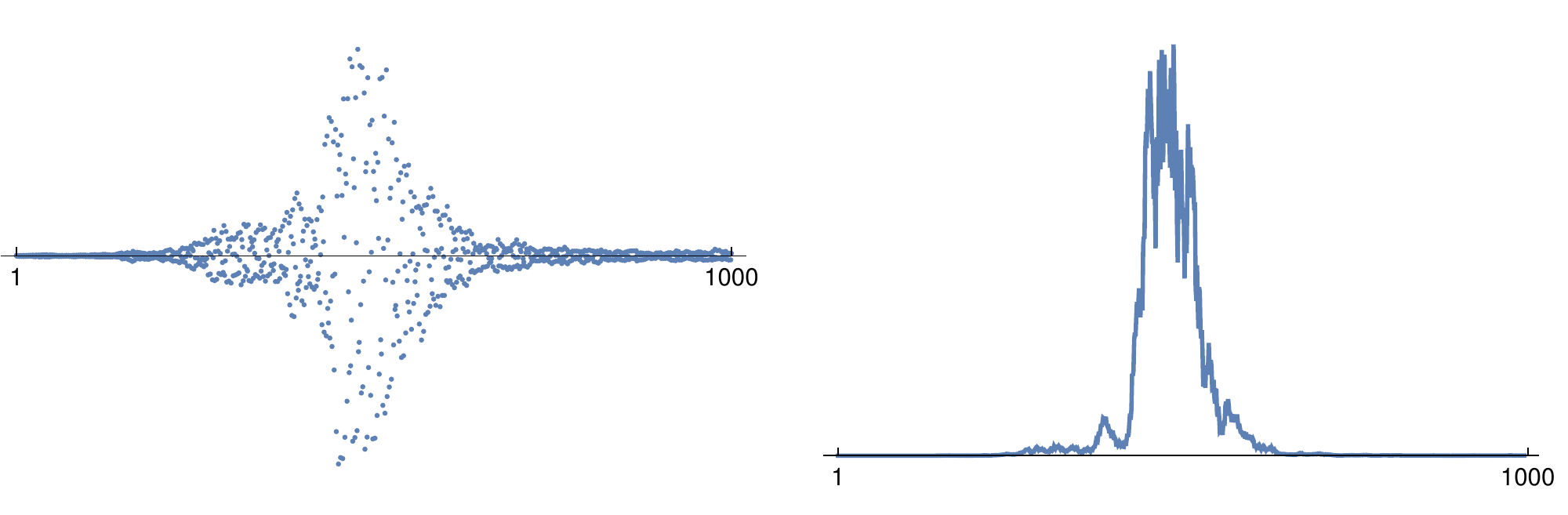}
\caption{An eigenvector of $H_{1000}$ with $\sigma = 8$ and its smoothed $\ell^2$ mass}
\end{figure}

Instead since the $\ell^2$ mass of the (normalized) eigenvector is a probability measure we use induced measure on $[0,1]$ as a natural description of its shape. For $\mu$ an eigenvalue of $H_n$ and $\psi^\mu$ the corresponding normalized eigenvector we study the measure on $[0,1]$ whose density is
$$
\abs{\psi^\mu \left( \nfloor \right)}^2 dt.
$$

We let $\mathcal{M}([0,1])$ be the space of finite measures on $[0,1]$ with the weak topology. By this we mean that $\mu_n \to \mu$ if $\int f d\mu_n \to \int f d\mu$ for every $f \in C_b([0,1], \R)$.

\noindent Our main result is a statement about the joint convergence in law of the pairs
$$
\left( \mu, \abs{\psi^\mu \left( \nfloor \right)}^2 dt \right) \in \R \times \cM[0,1]
$$
when we pick $\mu$ uniformly at random from the eigenvalues of $H_n$.

The asymptotic density near $E\in
(-2,2)$ is given by the arcsine law, $\frac{\rho}{2\pi}$ with
\begin{equation}\label{defrho}
\rho=\rho(E)=\frac{ 1 }{\sqrt{1-E^2/4}}   \mathbf{1}_{\abs{E} < 2}.
\end{equation}

\begin{theorem}\label{global_evector}
Let $E$ be distributed according to the arcsine law, $U$ uniform on $ [0,1]$, and $\cZ$ be a standard two-sided Brownian motion started from $0$ with $E$, $U$, and $\cZ$ independent.

Pick $\mu$ uniformly from the eigenvalues of $H_n$ and let $\psi^\mu$ be the corresponding normalized eigenvector.

Then letting $\tau(E) = (\sigma\rho(E))^2$ we have the following convergence in distribution
$$
\left( \mu, n  \abs{ \psi^{\mu} \left( \nfloor \right) }^2 dt  \right)
\Rightarrow
\left( E, \frac{S \bigl(\tau(E) (t-U)  \bigr) dt }{\int_0^1 ds \, S \left( \tau(E) (s -U) \right) } \right),
$$
where
$$
S(t) = \exp\left( \frac{\cZ_t}{\sqrt{2}} - \frac{\abs{t}}{4}  \right).
$$
\end{theorem}

The proof relies on the scaling limit of the transfer matrix framework for this problem that was developed in \citetalias{KVV}. Since $\cZ$ is independent of $U$ and $E$ we see that the shape $S(t)$ of the eigenfunction
does not depend on the corresponding spectral value, even though the scaling does. We think of $U$ as the peak of the eigenvector and it is uniform on $[0,1]$ due to the fact we chose the eigenvalue uniformly at random. So there was no bias as to where the peak should be. $\cZ$, $E$, and $U$ are all independent, which means all the components of the eigenvector: the shape, the spectral value, and the peak decouple.

\bigskip

The organization of this paper is the following. In Section \ref{TransferMatrix} we explain the transfer matrix framework along with the main theorem of \citetalias{KVV} along with our extension. In Section \ref{LocalConvergence}, we give a local version of Theorem \ref{global_evector}. And finally in Section  \ref{LocalImpliesGlobal} we show how this local result gives the proof of the main theorem.

\subsection*{Universality  conjectures}

We believe that the behavior described in Theorem \ref{global_evector} is universal for eigenvectors of {\it critical} random operators in one dimension in the Poisson limit. The first setting is very close to ours.
\begin{conjecture}
The scaling limit of the eigenvector picked from the spectral measure at zero of the infinite one-dimensional random Schr\"odinger operator is as described in Theorem \ref{global_evector}.
\end{conjecture}
We expect the same limit in several other one-dimensional cases.
\begin{conjecture}
The eigenvectors have a scaling limit as described in Theorem  \ref{global_evector} in the following settings:
\begin{enumerate}
\item for the critical continuous one-dimensional random Schr\"odinger operator, see \cite{Nakano};
\item for the Sine$_\beta$ operator as $\beta\to 0$, see \cite{Sineb};
\item for the stochastic Airy operator as $\beta \to 0$ for high eigenvalues, see \cite{RRV};
\item for the random Hill operator for high eigenvalues, see \cite{hill}.
\end{enumerate}
\end{conjecture}

\section{Transfer Matrix} \label{TransferMatrix}
\citetalias{KVV} showed that the transfer matrix framework has a limiting evolution; it is this limiting object that enabled them to characterize the limiting eigenvalue process.  Our main technical result is a slight strengthening of the convergence in that theorem. Our analysis will make use of this convergence and the correspondence between eigenvectors and transfer matrices. In order to state that theorem we first introduce the transfer matrix description of the spectral problem for $H_n$.

We can write the eigenvalue equation $H_n \psi = \mu \psi $ or
\begin{align*}
\psi({\ell-1}) + v_{\ell,n} \psi(\ell) + \psi({\ell+1}) &= \mu \psi(\ell),
\end{align*}
as the recursion $\psi({\ell+1})  =  (\mu - v_{\ell,n}) \psi(\ell) - \psi({\ell-1}) $ with $\mu$ an eigenvalue of $H_n$ when $\psi(0) = 0 = \psi(n+1)$.
We write this as,
\begin{equation}\label{txmat_form}
\bin{\psi({\ell +1})}{\psi(\ell) }=T(\mu-v_{\ell,n}
)\bin{\psi({\ell })}{\psi({\ell -1})}=M^\mu_{n}(\ell)
\bin{\psi_1}{\psi_0},
\end{equation}
where $$T(x):=\mat{x}{-1}{1}{0} \textrm{ and } M_{n}(\mu, \ell)
:=T(\mu-v_{\ell,n} )T(\mu-v_{\ell -1,n})\cdots T(\mu-v_{1,n}).$$
We also set $M_0(\mu,\ell)=I$ the identity matrix. Then $\mu$ is an eigenvalue of $H_n$ if and only if
\begin{equation}\label{ev_cond}
M_n(\mu, n) \bin{1}{0}=c \bin{0}{1},
\end{equation}
for some $c\in\R$ or, equivalently $(M_n(\mu, n))_{11}=0$.
Moreover, notice that the corresponding normalized eigenvector $\psi^\mu$ is given by
\begin{align}
\psi^\mu(\ell) = \frac{ (M_n(\mu,\ell-1))_{11}} {\sqrt{ \sum_{k=0}^{n-1}  (M_n(\mu,k))_{11}^2 } }, \quad \ell = 1, \dots, n.
\label{evector_condn}
\end{align}


For local analysis in view of (\ref{defrho}) we parameterize
$\mu=E+\frac{\la}{\rho(E) n}$. We will use the notation
$M_{n,E}(\la, \ell)$ to emphasize dependence on $\la$ and $E$, and use the
similar notation for other quantities.  Sometimes we will drop $E$ from our notation and when we do so we are implicity assuming that there is a fixed $E \in (-2,2)$ in the background.
Setting
\begin{equation}\label{epsn}
\eps_{\ell,n} =\frac{\la}{\rho
n}-\frac{\sigma\omega_\ell}{\sqrt{n}},
\end{equation}
we have
\begin{equation}\label{Mn}
M_{n,E}(\la, \ell) =T(E+\eps_{\ell,n} )T(E+\eps_{\ell -1,n})\cdots
T(E+\eps_{1,n}) \textrm{ for } 0\leq \ell \leq n.
\end{equation}

As $T(E + \epsilon_{\ell,n})$ is a perturbation of $T(E)$, we follow the evolution in the coordinates that diagonalize $T(E)$. For $\abs{E}< 2$,  we can write $T(E) = Z D Z^{-1}$ with
\begin{equation}
D=\mat{\zb}{0}{0}{z}, \; Z=\frac{i\rho(E)}{2} \mat{\zb}{-z}{1}{-1},\; Z^{-1}= \mat{1}{-z}{1}{-\overline{z}},\;
z=\frac{E}{2} + i\sqrt{1- \frac{E^2}4}.\label{zdef}
\end{equation}
The diagonalizing matrix $Z$ is unique up to right multiplication by a diagonal matrix.
As we will see, the limit of the transfer matrix evolution \eqref{txmat_form} is better understood in the basis given by $Z$.
Our choice
of $Z$ is so that  $(1,1)$ is mapped to the initial condition $(1,0)$ of the recursion \eqref{txmat_form}.   As a fractional linear fraction transformation, $Z$ maps the unit disk to the upper half plane,
mapping the triple $(1,0,-1)$ to $(\infty,z,0)$.

From this we can see that for $\abs{E}<2$, $M_{n,E}(\la,\ell)$ is a perturbation of the rotation matrix $D^\ell$ and so we cannot hope for a limiting process. However, if we regularize the evolution by undoing the rotation and consider instead
\begin{align}
Q_{n,E}(\la, \ell)= T^{-\ell}(E)M_{n,E}({\la}, \ell) \label{reg_transfer},
\end{align}
then we extend the results from \citetalias{KVV} as follows.

\begin{thm}\label{DiffusionTransfer} Assume $0<|E|<2$.
Let $\cB(t), \cB_2(t),\cB_3(t)$ be independent standard
Brownian motions in $\Real$,
$\cW(t)=\frac1{\sqrt{2}}(\cB_2(t)+i \cB_3(t))$. Then the
stochastic differential equation
\begin{equation}\label{LimitingTransfer}
d\X(\la,t)=\frac12  Z \of{ \mat{i \la}{0}{0}{- i \la} dt+
\mat{i d\cB}{ d\cW}{ d\overline\cW}{-i
d\cB}}Z^{-1}\X(\la,t), \qquad \X(\la,0)=I
\end{equation}
has a unique strong solution $\X(\la, t): \la\in\Comp,
t\ge 0$, which is analytic in $\lambda$.

\bigskip

\noindent Moreover, let $\tau = \left( \sigma \rho(E)  \right)^2$, then
\begin{equation*}
\left( \X_{n,E}\Bigl( \la, \left \lfloor nt/\tau\right \rfloor  \Bigr) , 0\leq
t\leq \tau \right) \Rightarrow (\X(\la/\tau, t),0\leq t \leq \tau),
\end{equation*}
in the sense of finite dimensional distributions for $\la$
and uniformly in $t$. Moreover, the random analytic functions $\X_{n,E}(\la,t)$ converge in law to
$\X(\la/\tau,t)$ with respect to the local uniform topology on $\C \times [0,\tau]$.
\end{thm}
\begin{remark}
This is an extension of the theorem proven in \citetalias{KVV}. The work we have done here is to strengthen the tightness argument which allows us to get convergence in law with respect to the local uniform topology on $\C \times [0,\tau]$. The extra tightness argument along with how this implies the result is in Section \ref{Tightness}.
\end{remark}

\section{Local Limits of Eigenvalue-Eigenvector Pairs} \label{LocalConvergence}

In this section we prove a local version of Theorem \ref{global_evector}. We will zoom in on the eigenvalue point process around a fixed $0< \abs{E}<2$. From Equation \eqref{defrho} we see that the eigenvalue spacings near $E$ are like $1/ (n \rho(E))$  and so we consider the operator $n \rho(E) (H_n -E)$ and its eigenvalues $\Lambda_{n,E}$.  Our local result is about the joint convergence of eigenvalue, eigenvectors pairs of this scaled operator. As with our global limit we consider the induced $L^2$ measure on $[0,\tau]$ coming from the eigenvector since it is otherwise too irregular to have a scaling limit. We think of these pairs as a point process on $X = \R \times \mathcal{M}[0,\tau]$,
$$
\mathcal{P}_{n,E} = \Bigl \{ \Bigl (  n \rho(E) ( \mu - E  )  + \theta , \frac{n}{\tau} \abs{\psi^\mu(\nfloortau)}^2 dt \Bigr )  :  \mu \text{ an eigenvalue of } H_n \Bigr \}.
$$
With the usual product topology $X$ is a complete, separable metric space. Let $\mathcal{M}(X)$ be the set of locally finite measures on $X$ with the local weak topology. In other words, we say $\mu_n \in \cM(X)$ converges to $\mu \in \cM(X)$ if for every continuous function $f:X \to \R$ with compact support, $\int f d \mu_n \to \int \psi d\mu$.  A random measure on $\cM(X)$ is a measurable map $\omega \rightarrow \mu \in \cM(X)$, with the Borel $\sigma$-algebra on $\cM(X)$. By the point process $\mathcal{P}_{n,E}$ we mean the random measure in $\mathcal{M}(X)$ given by the sum of the delta masses corresponding to points in the set. And by convergence in law of a sequence of point processes on $X$ we mean the usual notion of weak convergence of the corresponding random measures on $\mathcal{M}(X)$.

\begin{theorem} \label{local_evector}
Fix $0 < \abs{E} <2$ and take $\tau = \tau(E) = ( \sigma \rho(E))^2$. Let $\alpha$ be uniform on $[0, 2\pi]$. Then, the point process on $\R \times \mathcal{M}[0,\tau]$
$$
\Bigl \{ \Bigl (  n \rho(E) ( \mu - E  )  + \alpha , \frac{n}{\tau} \abs{\psi^\mu(\nfloortau)}^2 dt \Bigr )  :  \mu \text{ an eigenvalue of } H_n \Bigr \}
$$
converges in law to a point process $\mathcal{P}_E$.\\
\noindent Moreover, let $\mathcal{Z}$ be a  two sided Brownian motion started from $0$ and $U$ uniform on $[0,\tau]$ independent of $\mathcal{Z}$. For $t \in \R$, let
$$
S(t) = \exp \left( \mathcal{Z}_t/\sqrt{2} - \abs{t}/4 \right).
$$
Now define the measure $\mu_E$ on $X$ such that for every $F \in C_b \left(\R \times \mathcal{M}[0,\tau] \right)$,
$$
\int F(\la, \nu) \, d\mu_E(\la, \nu) = \frac{1}{2\pi} \int d \la \E F\left(\la, \frac{S(t-U ) dt}{ \int_0^\tau d s \, S(s-U ) }  \right).
$$
Then the intensity measure of $\mathcal{P}_E$ is $\mu_E$.
\end{theorem}
\begin{remark}
\citetalias{KVV} proved the convergence of the local eigenvalue point process and characterized the limit. Our result is an extension to the eigenvalue-eigenvector pairs.
\end{remark}
The proof of weak convergence proceeds in the usual steps. We first show subsequential convergence and then that the limit does not depend on the subsequence. We calculate the intensity measure in a separate lemma.

In order to characterize the limiting point process, we introduce two limiting random processes. Note that for $0< |E| < 2$, for any
$a,b\in \Real^2$ we have
\[
Z^{-1} \vect{a}{b}=\vect{
 a -b z }{\overline{a -b z }
},
\]
so $Z^{-1}$ maps real vectors to vectors with conjugate
entries. Since for $\lambda\in \Real$ the transfer matrix
$Q_{n,E}(\la,\ell)$ is real valued the process $Q(\la, t)$ will also
be real valued. Therefore, we can write for $\la \in \R$,
\begin{align}
%
\vect{ q^\la(t)}{\overline{
q^\la(t)}} := Z^{-1}\X(\la,t )\vect{1}{0} \label{Qtoq}
\end{align}
for some complex numbers $q^\la( t)$ where
$q^\la( 0 )= 1 $. We will show that $q^\la$ determines both the limiting eigenvalue point process and the limiting eigenvector shape.  It will be useful to write
$$(q^\la(t))^2 = r^\la(t) e^{i \theta^\la(t)}$$
in polar coordinates;  the branch of $\theta$ is chosen so that $\theta^\lambda(0)=0$ and $\theta$ is continuous in $t$.

\begin{lemma} \label{SDErtheta}
The quantities $r$ and $\theta$ are well defined and uniquely satisfy the stochastic differential equations,
\begin{align}
d \theta^\la(t) & =\la dt+ d\cB+ \im \left[
e^{-i\theta^\la( t)} d\cW \right],\quad \theta^\la( 0)=0 \\
d r^\la( t) & = \frac{dt}{4}  + \re \left[ e^{-i \theta^\la(t)} d\mathcal{W} \right], \quad r^\la(0) = 0.
\end{align}
coupled together for all values of $\lambda\in \Real$ where
$\cB$ and $\cW$ are independent standard real and complex Brownian
motions. \\
Moreover $\theta^\la(t)$ is almost surely real analytic in $\la$ and $\phi^\la(t) := \frac{\partial \theta^\la(t)}{\partial \la}$ satisfies the SDE
$$
d \phi^\la(t)=  dt - \re (e^{-i \theta^\la(t)} d\cW ) \phi^\la(t) .
$$
\end{lemma}

Our first step in proving Theorem \ref{local_evector} is to show convergence in law along subsequences. First define the process
$$
\sch^{\phi} = \left \{ \la \in \R: \theta(\la/\tau, \tau) \in   2 \pi \mathbb{Z} +  \phi \right \},
$$
and let $\sch^{*}=\sch^{U}$ with $U$ an uniform random variable in $[0,2\pi]$ independent of everything.

\begin{lemma}\label{subseq_convergence}
Fix $0 < \abs{E} < 2$. For $\la \in \R$, let $\bm{m}^\la_{n}$, $\bm{q}^\la$ be measures on $[0,\tau]$ with densities
\begin{align*}
d \bm{m}_{n}^\la(t) & = \abs{\Bigr( (2/\rho(E)) M_{n,E} \left(\la, \nfloortau \right) \Bigl)_{11}}^2 dt,\\
d \bm{q}^\la(t) & = \abs{q^\la(t)}^2 dt.
\end{align*}
Suppose that $n_j$ is a subsequence along which $z(E)^{n_j+1} \rightarrow z'$, see \eqref{zdef}. Then, in law,
$$
\Bigl\{ \left( \la ,  \bm{m}_{n}^\la  \right) :  {\la \in \Lambda_{n_j,E}}  \Bigr \}
\Rightarrow
\Bigl \{ \left( \la,   2 \bm{q}^{\la/\tau} \right ) : {\la \in \sch^{2\arg z'} }
\Bigr \}.
$$
\end{lemma}

The next lemma shows that the distribution of the limit does not depend on the subsequence.
\begin{lemma} \label{subseq_dist}
Fix $\tau >0$ and $U$ uniform in $[0, 2 \pi]$ independent of $\sch^\phi$. Then for any $\phi$ $\in \R$,
$$
\bigl \{ \left ( \la  + U, \bm{ q }^{\la/\tau}  \right):  {\la \in \sch^\phi} \bigr\}
=^d \left \{ \left( \la, \bm{ q}^{\la/\tau} \right ) : \la \in \sch^* \right \}.
$$
\end{lemma}

And finally we need the following lemma to help calculate the intensity measure of the limiting point process.
\begin{lemma} \label{intensity_measure}
Let $\cB$ be a standard Brownian motion started at zero, $U$ independent, uniform on $[0,\tau]$, and $f^u(t) = \frac{1}{2} (u - \abs{u-t})$.
Then for every $G \in C_b \left(\R \times C[0, \tau] \right)$,
$$
\E \sum_{\la \in \sch^*} G(\la, |q^{\la/\tau}|^2 )
= \frac{1}{2\pi} \int d \la \, \E \left[ G \left( \la,  \exp \left(  \frac{\cB}{\sqrt{2} } + \frac{f^U}{2} \right) \right) \right].
$$
\end{lemma}

The above three lemmas give the proof of Theorem \ref{local_evector}.
\begin{proof}[Proof of Theorem \ref{local_evector}]
Lemma \ref{subseq_convergence} gives that along a subsequence $n_j$ such that $z^{n_j}$ converges to $z'$, we have that
\begin{align*}
\Bigl\{ \left( \la+ \alpha ,  \bm{m}_{n}^\la  \right) :  {\la \in \Lambda_{n_j,E}}  \Bigr \}
& \Rightarrow
\Bigl \{ \left( \la + \alpha,   2 \bm{q}^{\la/\tau} \right ) : \la \in \sch^{2\arg z'}
\Bigr \}  \\
& =^d  \Bigl \{ \left( \la,   2 \bm{q}^{\la/\tau} \right ) : {\la \in \schs }  \Bigr \}
\end{align*}
with the equality following by Lemma \ref{subseq_dist}. Since from any subsequence we can extract a further subsequence $n_j$ such that $z^{n_j}$ converges, this gives that
$$
\Bigl\{ \left( \la+ \alpha ,  \bm{m}_{n}^\la  \right) :  {\la \in \Lambda_{n,E}}  \Bigr \}  \Rightarrow   \Bigl \{ \left( \la,   2 \bm{q}^{\la/\tau} \right ) : {\la \in \schs }  \Bigr \} .
$$
Now recall that for $\la \in \Lambda_{n,E}$, $\la = n \rho(E) (\mu - E)$  for $\mu$ an eigenvalue of $H_n$ and the corresponding normalized eigenvector is
\begin{align*}
\psi^\mu(\ell)= \frac{\left( M_{n,E}(\la, \ell)  \right)_{11}}{\sqrt{ \sum_{k=1}^n \abs{ \left( M_{n,E} (\la,k) \right)_{11} }^2 } }  , \quad \ell =1, \dots, n.
\end{align*}
And so since $d \bm{m}_n^\la(t) = \abs{\Bigr( (\rho/2) M_{n,E}(\la, \left( \nfloortau \right) \Bigl)_{11}}^2 dt$,
\begin{align*}
\frac{ n}{\tau} \abs{\psi^\mu (\nfloortau)}^2 dt
= \frac{d \bm{m}_n^\la(t)}{\bm{m}_n^\la[0,\tau]}.
\end{align*}
Since the function from $\cM[0,1]$ to itself given by $\mu \mapsto \mu/ \mu[0,1]$ is continuous except at zero and the probability that $\bm{m}_n^\la \equiv 0$ is zero, this gives the convergence in law,
\begin{align*}
\Bigl \{ \Bigl (  n \rho(E) ( \mu - E  )  + \theta , \frac{n}{\tau} \abs{\psi^\mu(\nfloortau)}^2 dt \Bigr )  :  \mu \text{ an eigenvalue of } H_n \Bigr \}
\\ \Rightarrow
\Bigl \{ \left( \la,   \frac{\bm{q}^{\la/\tau}}{\bm{q}^{\la/\tau}( [0,\tau] )   } \right ) : {\la \in \sch^{\tilde{\phi} } }
\Bigr \}.
\end{align*}
Now let $\cB$  be standard Brownian motion. Then $\cB_t+ \frac{1}{2} u $ has, up to a random constant addition, has the same distribution as $\cZ_{t-u}$ where $\cZ$ is a two-sided Brownian motion started from zero. Since the additive constants cancel in the normalization, we have
$$
\frac{\exp(\cB_t+ \frac{1}{2} (u - \abs{u-t})  )}{\int_0^\tau ds \exp \left(\cB_s+ \frac{1}{2} (u - \abs{u-s})  \right)}
=^d  \frac{ \exp \left( \cZ_{t-u} - |u-t| / 2 \right)}{\int_0^\tau  \exp \left( \cZ_{s-u} - |u-s|/2    \right)},
$$
as processes on $[0,\tau]$. And so from Lemma \ref{intensity_measure}, we have the intensity measure of the limiting point process.
\end{proof}

We now present the proofs of the three lemmas of this section.
\begin{proof}[Proof of Lemma \ref{subseq_convergence}] \label{local_evector_proof}
We will show convergence in law of random point measures on $X = \R \times \cM[0,\tau]$. In other words, we want to show that  $\mu_{n_j} =   \sum_{\la \in \Lambda_{n_j,E}} \delta(\la) \delta(\bm{m}_{n_j}^\la)$ converges in law to $ \mu = \sum_{\la \in \sch^{\tilde{\phi}}} \delta(\la) \delta({\bm{q}^{\la/\tau}})$ with respect to the local weak topology. By the general theory of point processes (see Proposition 11.1.VIII, \cite{DaleyJones})  it suffices to show that for any $h \in C_c(X, \R)$, the real valued random variables $\int h d\mu_{n_j}$ converge in law to $\int h d\mu$.

First, for all $w \in \C$, we let
\begin{align*}
F_n(w,t) & := \vectv{ F^1_n(w, t) }{F^2_n(w,t)} := Z^{-1} Q_{n,E}(w, \nfloortau) \eo, \\
F(w,t) & := \vectv{ F^1(w, t) }{F^2(w,t)} := Z^{-1} Q(w, t) \eo.
\end{align*}
By Theorem \ref{DiffusionTransfer}  we have that $Q_n(w, \nfloortau)$ converges in law with respect to the local uniform topology on $\C \times [0,\tau]$ (see Section \ref{Tightness})  to $Q(w/\tau, t)$. Since $Z^{-1}$ is a continuous deterministic transform, we also have that $F_n(w,t)$ converges in law to $F(w/\tau, t)$. We first show that $\mu_n$ is determined by $F_n$ while $\mu$ is determined by $F$.

Recall \eqref{reg_transfer} that
that we defined
$$
 Q_{n,E}(w,\ell) = T^{-\ell}(E)  M_{n,E}(w, \ell),
$$
and so
\begin{align}
\frac{2}{\rho(E)}  \left( M_{n,E}( w, \nfloortau) \right)_{11}
&= \eot \left( \frac{2}{\rho(E)} Z \right) D^{\nfloortau}  Z^{-1} Q_{n,E}(w, \nfloortau) \eo  \\
&= i\overline{z}^{\nfloortau + 1} F^1_n(w, t) - iz^{\nfloortau+1} F^2_n(w,t) \label{evector_transform}.
\end{align}
In other words $\bm{m}^\la_n$ is a function of $F_n$. Moreover, for $\la \in \R$, we have by Equation \eqref{Qtoq} that  $2 \abs{q^\la(t) }^2 = \abs{F^1(\la,t)}^2 + \abs{F^2(\la,t)}^2$ and so $\bm{q}^\la$ is a function of $F$.

Moreover,  $\Lambda_{n,E} = \{ w \in \R: \left( M_{n,E}( w, n) \right)_{11}= 0 \}$, which again is determined by $F_n$. And in fact, $(2/\rho) \left( M_{n,E}( w, n) \right)_{11}$ converges in law to
\begin{align*}
\tilde{m}(w)
& :=  \lim_{n_j \to \infty} i\overline{z}^{n_j +1} F^1_{n_j}(w, \tau) - iz^{n_j+1} F^2_n(w,\tau) \\
& = i\bar z'F^1(w/\tau, \tau) - iz' F^2(w/\tau, \tau).
\end{align*}
And now notice that for $\la \in \R$, by Equation \eqref{Qtoq}
\begin{align*}
\tilde{m}(\la) = 0
\iff \bar z' q(\la/\tau,\tau) - z' \overline{q(\la/\tau,\tau)} = 0
\iff \arg q(\la/\tau, \tau) \equiv \arg z' \quad (\rm{mod }\; \pi)
\end{align*}
In other words $\sch^{2\arg  z'}$  is the zero set of $\tilde{m}$, which is determined by $F$.


We have shown that $\int h d \mu_n$ is a measurable function of $F_n$ while $\int h d \mu$ is a measurable function of $F$. Since $F_n$ converges in law to $F$, the continuous mapping theorem (eg. \cite{Kallenberg}, Theorem 3.27) allows us to remove the randomness from the problem. We may assume that $F_n$ converges to  $F$ in the local uniform topology and simply show that this implies that $\int h d\mu_{n_j}$ converges to $\int h d \mu$.  We may also assume that $h= h_1 \cdot h_2$,  with $h_1 \in C_c(\C)$ and $h_2 \in C(\cM[0,\tau])$,

First notice that if $\la_n \to \la \in \R$, then as measures on $[0,\tau]$, $\bm{m}_{n}^{\la_n}$ converges weakly to $\bm{q}^{\la/\tau}$  (and so $h_2(\bm{m}_{n}^{\la_n})$ converges to $h_2(\bm{q}^{\la/\tau})$).  Take $u \in C[0,\tau]$, then
\begin{align*}
\int u \, d \bm{m}_{n}^{\la_n}
& =  \int_0^\tau u(t) \abs{\bar z^{\nfloortau+1}F_n^1(\la_n, t) + z^{\nfloortau+1} F_n^2(\la_n, t)}^2 dt.
\end{align*}
Expanding the absolute value, noting that $F_n(\la_n, t)$ converge uniformly on $[0,\tau]$ to $F(\la/\tau, t)$, and applying Theorem \ref{analytic_fact}  gives that
\begin{align*}
\lim_n \int u \, d \bm{m}_{n}^{\la_n}   & = \int_0^\tau u(t) \left ( \abs{F^1(\la/\tau,t)}^2 + \abs{F^2(\la/\tau, t)}^2 \right) dt \\
& = \int_0^\tau u(t) \, d \bm{q}^{\la/\tau}(t).
\end{align*}
Moreover since $F_n$ converges to $F$ and $z^{n_j+1}$ converges to $z'$, the analytic functions on $\C$, $(2/\rho) \left( M_{n_j,E}( w, n_j)) \right)_{11}$ converge in the local uniform topology to $\tilde{m}(w)$.  By Hurwitz's theorem this gives that the zeros of these functions converge pointwise.  And the real valued zeros converge to real valued zeros.  And so,
$$
\lim_{n_j} \sum_{\la \in \R: m_{n_j}(\la, \tau) =0} h_1(\la) h_2(\bm{m}_{n_j}^\la) = \sum_{\la \in \R: \tilde{m}(\la/\tau)=0 } h_1(\la) h_2(\bm{q}^{\la/\tau}),
$$
which completes the proof.
\end{proof}

\begin{proof}[Proof of Lemma \ref{subseq_dist}]
Recall that $r^\la= \log   \abs{q^\la}^2$. It therefore suffices to show that
$$
\left \{ \Bigl( \la  + U, r^{\la/\tau} \Bigr ) : {\la \in \sch^\phi} \right \}
=^d\left \{ \left( \la, r^{\la/\tau}  \right ) : {\la \in \schs} \right \}
$$
We first show that for $u \in \R$ fixed,
$$
\left \{ \la  + u, r^{\la/\tau} \right \}_{\la \in \sch^\phi}
=^d \left \{ \la , r^{\la/\tau}  \right\}_{\la \in \sch^{\phi + u}}
$$
Recall the SDEs from Lemma  \ref{SDErtheta},
\begin{align}
d \theta^\la & =\la dt+ d\cB+ \im \left[
e^{-i\theta^\la(t)} d\cW \right],\quad \theta^\la( 0)=0 \label{SDE1} \\
d r^\la & = \frac{dt}{4}  + \re \left[ e^{-i \theta^\la(t)} d\mathcal{W} \right], \quad r^\la( 0) = 0 \label{SDE2}.
\end{align}
coupled together for all values of $\lambda\in \Real$ where
$\cB$ and $\cW$ are standard real and complex Brownian
motions.  We  let $\tilde{\theta}^{\la}(t) := \theta^{\la - u/\tau}(t) + (u/\tau) t$ and $\tilde{r}^\la( t) := r^{\la- u/\tau}(t)$ and notice that $\tilde{\theta}^\la$ and $\tilde{r}^\la$ jointly solve Equations \eqref{SDE1} and \eqref{SDE2} with $d\tilde \cB= d\cB$ and $d\tilde \cW=e^{i(u/\tau)t} d \cW$.
\\
And so, since $\theta^{(\la-u)/\tau}(\tau) = \tilde{\theta}^{\, \la/\tau}(\tau) - u$,
\begin{align*}
\sch^\phi + u
&= \{ \la : \theta^{(\la- u)/\tau}(\tau) \in 2 \pi \Z + \phi \} \\
&  =   \{ \la : \tilde{\theta}^{\, \la/\tau} ( \tau ) - u \in 2 \pi \Z + \phi \}.
\end{align*}
Therefore
\begin{align*}
\left \{ \Bigl ( \la  + u, r^{\, \la/\tau}  \Bigr): {\la \in \sch^\phi}   \right\}
& = \left \{ \Bigl ( \la , r^{(\la-u)/\tau} \Bigr): {\la \in \sch^\phi + u }   \right\}\\
&=\left\{\Bigl( \la, \tilde{r}^{\, \la/\tau} \Bigr) : \tilde{\theta}^{\, \la/\tau} ( \tau ) \in  2\pi \Z + \phi +u \right \} \\
& =^d \left\{ \Bigl( \la, r^{\la/\tau} \Bigr) : {\la \in \sch^{\phi + u}} \right \}
\end{align*}
 by the uniqueness of solutions. Now if $U$ is uniform on $[0, 2 \pi]$, then $U+\phi \mod 2 \pi  $ is still uniform on $[0,2\pi]$ and so $\sch^{\phi + U} =^d \sch^*$ which finishes the proof.
\end{proof}

\begin{proof} [Proof of Lemma \ref{intensity_measure}]
Recall that $\sch^*= \{ \la : \theta^{\la/\tau}(\tau) \in 2 \pi \Z + U \}$,
where $U$ is uniform on $[0, 2\pi]$. Integrate out $U$ to get
\begin{align*}
\E \sum_{\la \in \sch^*} G(\la, r^{\la/\tau} )
&= \frac{1}{2\pi} \E \int_{0}^{2\pi } du \, \sum_{\la: \theta^{\la/\tau}(\tau)  \in 2\pi \Z + u} G(\la, r^{\la/\tau}) \\
& =  \frac{1}{2\pi} \E \, \int_{- \infty}^{\infty} du \sum_{\la: \theta^{\la/\tau}(\tau) = u} G(\la,r^{\la/\tau}).
\end{align*}
Now using Lemma \ref{SDErtheta} we have  that $\theta^{\la/\tau}(\tau)$ is almost surely a real analytic function in $\la$ and $r^{\la/\tau}$ and is continuous in $\la$. So we can apply the co-area formula and then Fubini to get
\begin{align}
\frac{1}{2\pi} \E \int_{-\infty}^{\infty} du \sum_{\la: \theta^{\la/\tau}(\tau) = u} G(\la, r^\la)
& = \frac{1}{2\pi} \int_{-\infty}^{\infty} d\la
\E \left[ G(\la, r^{\la/\tau} ) \abs{ \frac{\partial \theta^{\la/\tau}(\tau)}{\partial \la} } \right].   \label{coarea}
\end{align}

\noindent From Lemma \ref{SDErtheta}, we have that the evolution of $r^\la$ is given by
$$
d r^\la(t) = \frac{dt}{4}  + \re ( e^{-i \theta^\la( t)} d\mathcal{W} ).
$$
\noindent And moreover, $ \phi^{\la/\tau}(t) = \frac{\partial \theta^{\la/\tau}(t)}{\partial \la}$   is well defined, with SDE
$$
d \phi^{\la/\tau}=  \frac{dt}{\tau} - \re (e^{-i \theta^{\la/\tau}} d\cW ) \phi^{\la/\tau}.
$$
Now fix $\la$ and notice that $e^{-i \theta^{\la}} d\cW =^d d \cW$ and so the joint distribution of $r^{\la}$ and $\phi^{\la}$ does not depend on $\la$. We drop the $\la$ dependence and jointly solve for $r$ and $\phi$ to get

\begin{align*}
r_t & = \frac{t}{4} + \frac{\cB_t}{\sqrt{2}} \\
\phi_t &= \frac{1}{\tau }\int_0^t du e^{(r_u-r_t)}.
\end{align*}

And so by Fubini,
\begin{align*}
\E \left[ G(\la, r^{\la/\tau} ) \abs{ \frac{\partial \theta^{\la/\tau}(\tau)}{\partial \la} } \right]
= \frac{1}{\tau} \int_0^\tau du \E \left[ e^{ \left( r_u -r_\tau \right)} G(\la, r) \right],
\end{align*}

\newcommand{\tilp}{\tilde{phi}}
\newcommand{\tilr}{\tilde{r}}

Fix $u \in [0,\tau]$ and for simplicity, consider the process $\tilde{r}_t  = \cB_t + t/2$. This is just the time change $t \to 2t$.  We will calculate the distribution of the path $\tilr$ on $[0, \tau]$ weighted by $\exp\left( \tilr_{u} - \tilr_\tau \right)$.  In other words  if we take $\cR$ to be the law of $\tilr$ on $C[0,\tau]$, we need to characterize the measure on $C[0,\tau]$ given by,
$$
\exp(\omega_{u } - \omega_\tau) \text{d} \cR(\omega).
$$
By standard Girsanov theory, if we take $\cP$ to be the law of Brownian motion on $C[0,\tau]$, then $\text d\cR(\omega) =\exp \left( \frac{\omega_\tau}{2} - \frac{\tau}{8} \right) d\cP(\omega)$ and so
\begin{align}
\exp(\omega_{u} - \omega_\tau) d\cR(\omega) =
 \exp \left( \omega_{u} - \frac{\omega_\tau}{2} - \frac{\tau}{8} \right) d\cP(\omega).
\end{align}
Now if we let $x^{u}:=x^{u}(\omega)$ be the Brownian path reflected at ${u}$, we have that the corresponding exponential martingale of $x^{u}/2$ at $\tau$ is
\begin{align*}
\exp\left(  \frac{x^{u}_\tau}{2} - \frac{[x^{u}]_\tau}{8} \right)
=
\exp\left(\omega_{u} - \frac{\omega_\tau}{2} - \frac{\tau}{8}\right)
\end{align*}
where $[x^{u}]_t$ is the quadratic variation of $x^{u}$ at $t$. Therefore, by another application of Girsanov, if we let
$
f^{u}_t= [x^u/2,\omega]_t =\frac{1}{2} \left( {u} - \abs{{u}-t} \right),
$
then under the measure $\exp(\omega_{u} - \omega_\tau) d\cR(\omega)$ on $C[0,\tau]$ a path $\omega$ is distributed like $\cB + f^{u}$ where $\cB$ is a standard Brownian motion. Undoing the time change and applying Brownian scaling gives that,
$$
\E \left[ e^{ \left( r_u -r_\tau \right)} G(\la, r) \right]
= \E \left[ G \left( \la,  \frac{\cB}{\sqrt{2} } + \frac{f^u}{2} \right) \right],
$$
which shows
$$\E \sum_{\la \in \sch^*} G(\la, r^{\la/\tau} ) =
\frac{1}{2\pi} \int_{-\infty}^{\infty} d\la
\frac{1}{\tau} \int_0^\tau du
\E \left[ G \left( \la,  \frac{\cB}{\sqrt{2} } + \frac{f^u}{2} \right) \right].
$$
Treating the integral over $u$ as an expectation, and using the continuous mapping theorem for the map of functions $f\mapsto \exp(f)$ we get the claim.
\end{proof}

\begin{proof} [Proof of Lemma \ref{SDErtheta}]
We let $X(\la,t) = Z^{-1}Q(\la,t)$.  From  Equation \eqref{LimitingTransfer} we have the following stochastic differential equation for $X$ in $t$,
\begin{align*}
dX(\la,t)=\frac12  \of{ \mat{i \la}{0}{0}{- i \la} dt + \mat{i d\cB}{ d\cW}{ d\overline\cW}{-i
d\cB}}X(\la,t), \qquad X(\la, 0)=Z^{-1}.
\end{align*}
 This gives that
\begin{align*}
d X_{11}(\la,t) & =\frac{i \la}{2} X_{11}(\la,t) \, dt +    i X_{11}(\la,t) d\cB +  X_{21}(\la,t) d\cW.
\end{align*}
If $\la \in \R$, then $X(\la,t)_{11} = \overline{X(\la,t)_{21}}$ and moreover  $q^\la(t) =  X(\la,t)_{11}$. We fix $\la \in \R$ and drop it from our notation to get
\begin{align*}
d q & =\frac{i \la}{2} q \, dt +  \frac{1}{2} \left( i q d\cB + \overline{q} d\cW \right) \quad q(0) = 1
\end{align*}
Ito's formula then gives that
\begin{align*}
d \log q
&= \frac{dq}{q} - \frac{1}{2} \frac{(dq)^2}{q^2} \\
& = \frac{i\la}{2} dt + \frac{i}{2} d\cB + \frac{1}{2}  \frac{\overline{q}}{q} d\cW + \frac{dt}{8}
\end{align*}
Since $r = 2 \re \log  q $ and $\theta = 2 \im \log  q$, this yields for $\la \in \R$, the following SDEs in $t$,
\begin{align*}
d r  & =  \re \left( \frac{\overline{q}}{q} d\cW \right) + \frac{dt}{4}, \\
d \theta &=  \la dt +  d\cB +  \im \left(\frac{\overline{q}}{q} d\cW \right).
\end{align*}
Noting that $\frac{\overline{q}}{q} = \exp(- i \theta )$ finishes the proof.
\end{proof}

\section{Proof of Theorem \ref{global_evector}} \label{LocalImpliesGlobal}
We are now in a position to prove the main theorem of the paper. We will average the local result of Theorem \ref{local_evector} to get the more macroscopic version of the theorem. In order to do so we need to be able to control the number of eigenvalues in an a microscopic interval (of size $1/(\rho n) )$ around $E$. We will the need the following lemma whose proof is given in Section \ref{LocalEvalueEstimate}.

\begin{lemma} \label{mom_num_evalues}
Fix $R> 0$ and let $\Delta_n(E) = \left( E - \frac{R}{n \rho(E)}, E + \frac{R}{n \rho(E)} \right)$. Furthermore, let $N_n(E) = \abs{\Lambda_n \cap \Delta_n(E)}$ be the number of eigenvalues of $H_n$ in $\Delta_n(E)$. Then for any $\epsilon > 0$,
$$
\sup_n \sup_{E \in (-2 + \epsilon, 2 - \epsilon)} \E \left[ N_n(E) \right]^{3/2} < \infty.
$$
\end{lemma}

We now give the proof of Theorem \ref{global_evector}.
\begin{proof} [Proof of Theorem \ref{global_evector}]
Take $\theta$ uniform on $[0, 2\pi]$ and let $\bm{\psi}_{n}^\mu \in \cM[0,1]$ be the measure with density $\abs{\psi^\mu \left( \nfloor \right)}^2$. Using Theorem \ref{local_evector}  and the time change $t \to \tau t$, we have that  for $0 < \abs{E} <2$, the point process
$$
\cP_{E,n} = \left \{ \bigl( n\rho(E)(\mu-E) + \theta, n \bm{\psi}_{n}^\mu \bigr ) : {\mu \in \Lambda_n} \right \}.
$$
converges in law to a limiting point process $\cP_\tau$.

\noindent In particular, if we fix $g_1  = (1- \abs{x}) \mathbf{1}_{[\abs{x} \leq 1]}$, $g_2 \in C_b(\R \times \mathcal{M}[0,1])$ and let
$$
G_n(E) := \sum_{\mu \in \Lambda_n} g_1\Bigl(  n\rho(E)(\mu-E)  \Bigr) g_2\left(\mu, \bm{\psi}_{n}^\mu \right).
$$
Then for fixed $\abs{E} < 2$, $G_n(E)$ converges in distribution to $G(E)$ and
\begin{align}
\E G(E)
= \frac{1}{2\pi} \,  \E g_2 \left( E, \frac{S(\tau( t-u) ) dt}{ \int_0^1 d s \, S(\tau(s-u) ) }  \right).
\label{limit_expect}
\end{align}

We now show that $\int \E  G_n(E) d \rho(E)$ converges to $\int \E G(E) d\rho(E)$ from which the result will follow.

\noindent Fix $\epsilon >0$. Since $\supp \, g_1 \subset [-1,1]$, we let
$$
N_{n}(E) = \left|\left \{ \mu \in \Lambda_n : \abs{\mu -E} \leq 1/(n \rho(E) \right \}\right|,
$$
which gives that $G_n(E) \leq \norm{g_1}_{\infty} \norm{g_{2}}_{\infty} N_{n}(E)$. And so from Theorem \ref{mom_num_evalues},
$$
\sup_n \sup_{0 < \abs{E} < 2 - \epsilon } \E \left[ G_n(E) \right]^{3/2} < \infty.
$$
Therefore $G_n(E) \mathbf{1}_{\abs{E} < 2 - \epsilon }$ is uniformly integrable with respect to $\P \times d \rho$. And so since $G_n(E)$ converges in law to $G(E)$, we have that
\begin{align}
\lim_{n \to \infty} \int d \rho(E) \E \Bigl[ G_n(E) \mathbf{1}_{\left[ \abs{E} < 2- \epsilon \right]} \Bigr]
=  \int d \rho(E) \E \Bigl[ G(E) \mathbf{1}_{\left[ \abs{E} < 2- \epsilon \right]} \Bigr]. \label{average_convergence_epsilon}
\end{align}
Now by Fubini,
\begin{align*}
\int d \rho(E) \E \Bigl[ G_n(E) \mathbf{1}_{\left[ \abs{E} < 2- \epsilon \right]} \Bigr]
& = \E \sum_{\mu \in \Lambda_n} g_2\left(\mu, \bm{\psi}_{n}^\mu \right) \int_{-2+\epsilon}^{2- \epsilon}  d\rho(E) g_1\Bigl(  n\rho(E)(\mu-E)  \Bigr).
\end{align*}
Fix $\delta> \epsilon$ and let $A_{n}(\delta) = \{ \mu \in \Lambda_n : \abs{\mu} < 2- \delta \}$, $ B_n(\delta) =  \{ \mu \in \Lambda_n : \abs{\mu} \geq 2- \delta \}$.
We write
$$
\int d \rho(E) \E \Bigl[ G_n(E) \mathbf{1}_{\left[ \abs{E} < 2- \epsilon \right]} \Bigr] = \E \left[ \sum_{\mu \in A_n(\delta)} g(\mu) \right]+ \E \left[ \sum_{\mu \in B_n(\delta)} g(\mu) \right],
$$
with
$$
g(\mu) = g_2\left(\mu, \bm{\psi}_{n}^\mu \right) \int_{-2+\epsilon}^{2- \epsilon}  d\rho(E) g_1\Bigl(  n\rho(E)(\mu-E)  \Bigr),
$$
and deal with each piece separately.

\bigskip

By Lemma \ref{l:dos} we have
\begin{align}
\lim_{n \to \infty} \frac{1}{n} \E \abs{ B_{n}(\delta)}
  =\frac{1}{\pi} \int_{2-\delta}^2 \rho(s)ds \leq C \sqrt{\delta}. \label{Bdelta_bound}
\end{align}

\noindent Now use Lemma \ref{rho_integral2} to get that for $\mu \in B_n(\delta)$,
\begin{align*}
\int  d\rho(E) g_1\Bigl(  n\rho(E)(\mu-E)  \Bigr) \leq \frac{D}{n}.
\end{align*}
And along with equation \eqref{Bdelta_bound} this gives
\begin{align*}
\E \sum_{\mu \in B_n(\delta)}g(\mu)
& \leq \norm{g_2}_{\infty}\frac{D}{n} \abs{B_n(\delta)} \\
 & =  O(\sqrt{\delta}).
\end{align*}
Now for $n$ large enough if $\mu \in A_n(\delta)$,
\begin{align*}
\int_{-2 + \epsilon}^{2 - \epsilon} g_1\Bigl(n\rho(x)(x-\mu) \Bigr) d\rho(x)
& = \int_{-2}^2  g_1\Bigl(n\rho(x)(x-\mu) \Bigr) d\rho(x) \\
& = \frac{1}{n} \int g_1(x) \, dx  + o\left( 1/n \right) \\
& = \frac{1}{n} + o\left( 1/n \right)
\end{align*}
    The first equality follows from the fact that for $x \in [-2,2]$, $\rho(x) \geq 1$. And so since $g_1 \in C_c(\R)$, we have that $\abs{x-\mu} \leq D/n$ for some constant $D$. Since $\mu < 2 - \delta$, we have that $\abs{x} < 2 - \epsilon$ for $n$ large enough. The second equality follows from Lemma \ref{rho_integral}.  And so
\begin{align*}
\E \sum_{\mu \in A_n(\delta)}g(\mu)
& = \frac{1}{n} \sum_{\mu \in A_n(\delta)} \E g_2 \bigl(\mu, \bm{\psi}^\mu \bigr) + o(1) \\
& = \frac{1}{n} \sum_{\mu \in \Lambda_n} \E g_2 \bigl(\mu, \bm{\psi}^\mu \bigr) + O(\sqrt{\delta}) + o(1),
\end{align*}
with the last equality coming from equation \eqref{Bdelta_bound}.
To sum up
\begin{align}
\int d \rho(E) \E \Bigl[ G_n(E) \mathbf{1}_{\left[ \abs{E} < 2- \epsilon \right]} \Bigr]
= \frac{1}{n} \sum_{\mu \in \Lambda_n } \E g_2(\mu, \bm{\psi}^\mu) + o(1) + O(\sqrt{\delta}). \label{average_average}
\end{align}

On the other hand,
\begin{align*}
\int d \rho(E) \E \Bigl[ G(E) \mathbf{1}_{\left[ \abs{E} < 2- \epsilon \right]} \Bigr]
& =\int d \rho(E) \E \left[ G(E) \right]  + O(\epsilon).
\end{align*}

And so by equation \eqref{limit_expect} along with equation \eqref{average_average} and the convergence from equation \eqref{average_convergence_epsilon} we have that
$$
\lim_{n \to \infty} \frac{1}{n} \sum_{\mu \in \Lambda_n} \E \,  g_2 \bigl(\mu, \bm{\psi}^\mu \bigr)
= \frac{1}{2 \pi}  \int d \rho (E) \E g_2 \left(E, \frac{S(\tau (t-u) ) dt}{ \int_0^1 d s \, S(\tau (s-u) ) }  \right)+ O(\epsilon) + O(\delta).
$$
Since $\delta > \epsilon$ was arbitrary, this completes the proof.
\end{proof}

\section{Tightness} \label{Tightness}
\newcommand{\zx}{X}
\newcommand{\dzx}{{\mathcal{Y}}}

In this section we discuss the underlying tightness bounds we need to prove the weak convergence in Theorem \ref{DiffusionTransfer}.

We will use the following notions of convergence. Let $\mathcal{A}_d$
denote the space of continuous functions from $\C \times [0,1]$ to $\Comp^d$ that are also analytic in the first variable. In other words, if $f \in \cA_d$, then for every $t \in [0,1]$, $f(\cdot, t)$ is an analytic function from $\Comp$ to $\Comp^d$. We equip $\cA_d$ with
the metric
\[
d(f,g):=\sum_{r=1}^{\infty}
2^{-r}\frac{\norm{f-g}_r}{1+\norm{f-g}_r}, \quad
\quad\norm{h}_r:=\max_{( z,x) \in D_r}\norm{h(z,x)},
\]
where $D_r = B_r \times [0,1]$ and $B_r = \{ w \in \C : |w| \leq r\}$. Under this metric the space $\cA_d \subset C \left( [0,1] \times \C, \C^d \right)$  is a complete, separable metric space.

\bigskip

A
random function in $\cA_d$ is a measurable mapping
$\omega\To f \in \cA_d$ from a probability space
$(\Omega,\cF,P)$ to $(\cA_d,\cB)$, where $\cB$ is the Borel
$\sigma$-field generated by the metric $d$. The law of $f$
is the induced probability measure $\rho_f$ on
$(\cA_d,\cB_d)$. A sequence $f_\ell$ of random
analytic functions is said to converge in law to a random
 $f \in \cA_d$ if $\rho_{f_\ell }\To \rho_f $
in the usual sense of weak convergence.

\begin{prop}\label{convergence_analytic_functions} Suppose
$f_\ell$ is a sequence of random functions
in $\cA_d$ such that
\newline\noindent{\rm(1)} For every $w \in \C$, the processes $f_\ell(w, \cdot) \in C\left( [0,1], \C^d \right)$ are tight,
\newline\noindent{\rm(2)} For every $ r >0$,
\begin{align}
\lim_{M \to \infty } \sup_\ell  \P \left(  \norm{f_\ell}_r > M \right) & =0, \label{bound_tight}
\end{align}
\newline\noindent{\rm(3)} {For each $m\geq 1$ and $(z,t)=((z_1, t_1),(z_2, t_2),\cdots,(z_m,t_m) )\in \left( \C \times [0,1] \right)^m$ there is a probability distribution $\nu^{(z,t)}_m$ on $( \C ^d)^m$ and the random vector
$(f_\ell (z_1, t_1),f_\ell (z_2, t_2),\cdots,f_\ell
(z_m, t_m))\in \left( \Comp^d \right)^m$ converges in law to
$\nu^{z,t}_m$.}
\newline \noindent
Then there is a random function $f$ in
$\cA_d$  such that $f_\ell$ converges in law to
$f$. Moreover for each
$(z,t)=((z_1, t_1),(z_2, t_2),\cdots,(z_m,t_m) )\in \left( \C \times [0,1] \right)^m$,
$(f(z_1,t_1),f(z_2, t_2),\cdots,f(z_m, t_m))\in \Comp^m$
has distribution $\nu^{(z,t)}_m$.
\end{prop}

\begin{proof}
We first show that Assumptions \textrm{(1)} and \textrm{(2)} imply that the sequence $f_\ell$ is tight. We may assume that each $f_\ell \in \cA_1$ since tightness in every coordinate function implies that the sequence is tight.
\newline \noindent Fix $r>0$, $|w|$, $|u| \leq r$, and take $f \in \cA_1$. Then, by Cauchy's integral formula,
\begin{align*}
f(w,t) - f(u,t)
& = C_r \int_{|z| = 2r }
\left( \frac{f(z,t)}{w-z}  - \frac{f(z,t)}{u-z}\right) dz \\
& = C_r \int_{|z| = 2r } \frac{f(z,t)}{(w-z)(u-z)} \left( u-w \right)  dz
\end{align*}
And so Jensen's inequality along with the fact that $\abs{z-u}$, $\abs{z-w}$ $\geq r$ gives that, for every $t$,
\begin{align*}
\abs{f(w,t) - f(u,t)} \leq C_r  \norm{f}_{2r}  \abs{u-w}.
\end{align*}
This inequality gives that for $\abs{\zeta} \leq r$,
\begin{align*}
\abs{f(u,t) - f(w,s)} \leq  C_r \norm{f}_{2r} \left( \abs{u-\zeta} + \abs{w-\zeta} \right) + \abs{f(\zeta,t) - f(\zeta,s)}.
\end{align*}
And so if we take any $\alpha$-net $K_\alpha \subset B_r$ and take $\delta  < \alpha/2$,
\begin{align}
\sup_{\substack{\norm{ (w,t) - (u,s)} < \delta  \\  |w|, |u| \leq r } } \abs{f(w, t) - f(u, s)}
& \leq 2 C_r \norm{f}_{2r} \alpha + \max_{w \in K_\alpha} \sup_{|s-t|<\delta}  \abs{f(w,t) - f(w,s)}. \label{net_ineq}
\end{align}

\noindent Now fix $\epsilon > 0$. Since $f_\ell(w, \cdot)$ is tight for $w \in \C$, for every $\gamma>0$ we can find a $\delta_w >0$ such that
$$
\sup_{\ell \in \N} \mathbb{P} \left( \sup_{|s-t| < \delta} \abs{f_\ell(w,t)- f_\ell(w,s)} > \epsilon \right) < \gamma.
$$
In fact, just by adding probabilities, for any $\gamma, \alpha>0$ we can find a finite $\alpha$-net $K_\alpha \subset B_r$ and a $\delta_\alpha>0$ such that,
\begin{align}
\sup_{\ell \in \N} \mathbb{P} \left( \max_{w \in K_\alpha} \sup_{|s-t| < \delta_\alpha} \abs{f_\ell(w,t)- f_\ell(w,s)} > \epsilon \right) < \gamma.
\label{dense_tight}
\end{align}
Now fix $\gamma >0$. Assumption \textrm{(2)} means that we can find an $M$ such that $\mathbb{P} \left(  ||f_\ell||_{2r} > M \right) < \gamma$. Take $\alpha <  \epsilon (2MC_r)^{-1}$ and find a finite $\alpha$-net $K_\alpha$ and a $\delta_\alpha$ satisfying Equation \eqref{dense_tight}. Finally take $\delta = \min(\delta_\alpha, \alpha/2)$. Using Equation \eqref{net_ineq}, we get that,
\begin{align}
\sup_{\ell \in \N} \mathbb{P} \left( \sup_{\substack{\norm{ (w,t) - (u,s)} < \delta  \\  |w|, |u| \leq r } } \abs{f_\ell(w,t) - f_\ell(u,s)} \geq 2 \epsilon \right) < 2 \gamma.
\end{align}
Since $\epsilon$ and $\gamma$ were arbitrary, this inequality along with Assumption \textrm{(2)} and Arzel\`{a}-Ascoli gives tightness of the sequence $f_{\ell}$ restricted to the discs $D_r$. And so by Prokohorov's theorem a subsequence of $f_\ell $ restricted to $D_r$ converges in law. By a diagonal argument, there is a subsequence $f_{\ell_k}$ such that for each integer $r$, the restriction of $f_{\ell_k}$ to $D_r$ converges to a random analytic function $f_r$ on $D_r$. The distributions of the functions $f_r$ are consistent with respect to restricting to smaller discs, and thus there is a random analytic function $f$ on $\C \times [0,1]$ such that $f_{\ell_k}\To f$ in law with respect to the local uniform topology. Condition $(3)$ is strong enough to ensure that $f$ is unique and thus $f_\ell \To f$ in law.
\end{proof}

\begin{proof} [Proof of Theorem \ref{DiffusionTransfer}]
We intend to apply Lemma \ref{convergence_analytic_functions} to $Q_n(w,t) := Q_{n, E}( w, {\nfloortau})$. We cannot apply this directly since for any $w \in \C$, the processes $Q_n(w, \cdot)$ are piecewise constant but not continuous. Instead, for all $w \in \C$ we let $\tilde{Q}_n(w, \cdot)$ be the linearized version of the process $Q_n(w, \cdot)$. By this we mean the function whose graph is given by the straight line between each consecutive jump discontinuity of $Q_n(w, \cdot)$. Since $Q_n$ are analytic for any fixed $t$, $\tilde{Q}_n  \in \cA_4$ .  Theorem 1  of \citetalias{KVV} gives the tightness bound \textrm{(2)} for $\tilde{Q}_n$. Theorem \textrm{2} of \citetalias{KVV} and the continuous mapping theorem gives that for fixed $w \in \C$, $\tilde{Q}_n(w, \cdot)$ converge in law with respect to the uniform topology and so by Prokhorov the tightness bound \textrm{(1)} follows. This theorem also gives convergence of the finite dimensional distributions of $Q_n$ and hence those of $\tilde{Q}_n$ which is condition \textrm{(3)}. So by Lemma \ref{convergence_analytic_functions} $Q_n$ converges in law to $Q$ and since $d(Q_n, \tilde{Q}_n)$ goes to zero in probability we get that $Q_n$ converges in law to $Q$ with respect to the local uniform topology.

Note that the matrix $Z$ used here differs from the $Z$  in \citetalias{KVV} by right multiplication by constant times a diagonal matrix with diagonal entries $1$, $-1$. Conjugation by this matrix leaves the limiting SDE invariant.
\end{proof}

\section{Local Eigenvalue Estimate} \label{LocalEvalueEstimate}
In this section we give the proof of Lemma \ref{mom_num_evalues}. The moment bound on the number of eigenvalues in a macroscopic interval follows from an application of  the following.
\begin{theorem} [\cite{SimonLast}, Theorem 2.2]
Let $\mu < \mu'$ be consecutive eigenvalues of $H_n$. Then for any $E \in (\mu, \mu')$,
\begin{align}
\mu' - \mu \geq \left( \sum_{\ell=1}^n \norm{M_{n}(E, \ell)}^2 \right)^{-1} .  \label{evalue_space_lb}
\end{align}
\end{theorem}

\begin{corollary} \label{num_evalue_transfer}
Let $N$ be the number of eigenvalues of $H_n$ in the finite interval $\Delta$, and let $\beta>0$. Then
$$
((N-1)^+)^{1+\beta} \leq  (n\abs{\Delta})^{\beta} \int_{\Delta} \sum_{\ell=1}^n \norm{M_{n}(E, \ell)}^{2+2\beta} \, dE.
$$
\end{corollary}
\begin{proof} Assume $N\ge 2$ as otherwise the claim is trivial.
Let $\tau(E) := \sum_{\ell=1}^n \norm{M_{n}(E, \ell)}^2$.
Let $\mu_1<\ldots <\mu_N$ be the consecutive eigenvalues of $H_n$ in $\Delta$.
By Theorem \ref{LocalEvalueEstimate} we have
$$
(\mu_{i+1}-\mu_i)^{-\beta} = \int_{\mu_i}^{\mu_{i+1}} (\mu_{i+1}-\mu_i)^{-\beta-1} dE \le \int_{\mu_i}^{\mu_{i+1}}\tau(E)^{\beta+1}dE.
$$
Summing this over $i$ and applying Jensen's inequality yields
$$
(N-1)^{\beta+1}\left(\sum_{i=1}^{N-1} \mu_{i+1}-\mu_i\right)^{-\beta} \le \int_\Delta \tau(E)^{\beta+1}dE.
$$
The sum of the right equals $\mu_N-\mu_1\le |\Delta|$, so the claim follows by another application of Jensen's inequality.
\end{proof}

To prove Theorem \ref{mom_num_evalues} via Corollary \ref{num_evalue_transfer} we need a moment bound on the transfer matrices.
\begin{lemma} \label{norm_process_bound}
Let $\norm{\cdot}$ be the Hilbert-Schmidt norm on $M_{2 \times 2}(\C)$. There is a continuous function $f$ on $(-2,2)$ such for every $E \in (-2,2)$,
$$
\sup_n \max_{0 \leq \ell \leq n} \E \norm{M_{n}(E, \ell) - I}^3  < f(E).
$$
\end{lemma}

\begin{proof}
Fix $E \in (-2,2)$ and $n \in\N$ and recall that for $ 0 \leq \ell \leq n$,
$$ M_{n}(E, \ell)=T(E-v_{\ell,n} ) T(E-v_{\ell -1,n}) \cdots T(E-v_{1,n}), $$
with $T(x) :=\mat{x}{-1}{1}{0}$ and
$v_{\ell,n} =\frac{\sigma\omega_\ell}{\sqrt{n}}$.\\

Abbreviate $T=T(E)$.
We will prove a bound for the process $X_\ell = T^{-\ell} M_{n}(E, \ell)$.
Using the identity
$$
T(y)T^{-1}(x ) =  I + \mat{0}{y-x}{0}{0},
$$
we have that
\begin{align}
X_\ell
&= T^{-\ell} T(E-v_{\ell,n}) T^{-1} T^\ell  X_{\ell-1} \\
&=( I - v_{\ell,n} \cE_\ell ) X_{\ell-1}, \label{mtx_recursion}
\end{align}
where $\cE_{\ell}= T^{-\ell} \mat{0}{1}{0}{0}T^{\ell}$.\\
We first show that
\begin{align}
\norm{\cE_\ell} \leq c_1 (\rho(E))^2, \label{cE_bound}
\end{align}
where $c_1$ does not depend on $n$ or $E$ and $\rho(E) =  1/ \sqrt{1 - (E/2)^2}$. Recall that we can write $T(E) = Z D Z^{-1}$ where
\begin{equation}
D=\mat{\zb}{0}{0}{z},\quad  Z=\frac{i\rho(E)}{2} \mat{\zb}{-z}{1}{-1},\quad Z^{-1}= \mat{1}{-z}{1}{-\overline{z}}.
\end{equation}
with $z=E/2 + i\sqrt{1- (E/2)^2}$. \\
Using the submultiplicativity of the Hilbert-Schmidt norm along with the fact that $\abs{z} =1$ gives that
\begin{align}\label{Tellbound}
\norm{T^{\ell}(E) } \leq 16 \rho(E)\, \quad \text{for every }\ell \in \Z.
\end{align}
And since $\norm{\cE_\ell} \leq \norm{T^{\ell}(E) } \norm{T^{-\ell}(E)}$, we get the bound \eqref{cE_bound}.

\bigskip

Now notice that $X_\ell$ is a martingale with $X_0 = I$. We use the Burkholder-Davis-Gundy inequality along with Doob's Decomposition to get that for $0 \leq \ell \leq n$,
\begin{align*}
\E \max_{k \leq \ell} \norm{X_k - I}^3
& \leq c_2 \, \E \left(\sum_{k=1}^{\ell} \E \left[  \norm{X_k -X_{k-1}}^2    | \mathcal{F}_{k-1} \right]  \right)^{3/2},
\end{align*}
Now use that $X_k - X_{k-1} = v_k \cE_k X_{k-1}$, the bound on $\cE_k$, and that $\E v_{\ell,n}^2 = \sigma^2/n$ to get that
\begin{align*}
\E \max_{k \leq \ell} \norm{X_k - I}^3
& =  c_2 \, \E \left(  \frac{c_1 \sigma^2 \rho(E)^2}{n} \, \sum_{k=1}^{\ell} \norm{X_{k-1}}^2   \right)^{3/2} \\
& \leq c_3 \rho(E)^{3}  \frac{1}{n} \, \E \sum_{k=1}^{\ell}  \norm{X_{k-1}}^3,
\end{align*}
with the last inequality following from Jensen. Now using the inequality $\norm{A + B}^p \leq 2^p(\norm{A}^p + \norm{B}^p)$,
\begin{align}
\E \max_{k \leq \ell} \norm{X_k - I}^3
& \leq \frac{c_3 \rho(E)^3}{n}  \sum_{k=1}^{\ell} \left( \E \norm{X_{k-1}-I}^3 + \norm{I}^3  \right)   \\
& \leq c_4  \rho(E)^3  \left( 1 + \frac{S_{\ell-1}}{n} \right), \label{max_norm_bound}
\end{align}
where we have set $S_\ell = \sum_{k=1}^\ell \E \norm{X_k-I}^3$. This gives that
\begin{align*}
S_\ell - S_{\ell-1}
&= \E \norm{X_\ell - I}^3 \\
& \leq c_4 \rho(E)^3 \left( 1 + \frac{S_{\ell-1}}{n} \right),
\end{align*}
Finally, letting $R_\ell = 1 + S_\ell/n$, we have that
$
R_\ell \leq R_{\ell-1} (1 + c_4 \rho(E)^3/n),
$
and so $R_\ell \leq \exp(c \rho(E)^3)$ for $1 \leq \ell \leq n$. Therefore, equation \eqref{max_norm_bound} gives that
\begin{align*}
\E \max_{0 \leq k \leq n} \norm{X_k - I}^3
& \leq c_4 \rho(E)^3 R_{n-1}\\
& \leq d_1 \rho(E)^3 \exp(d_2 \rho(E)^3),
\end{align*}
for some constants $d_1$ and $d_2$ that do not depend on $E$ or $n$. Since $M_{n}(E, \ell) = T^{-\ell}(E) X_\ell$, the bound \eqref{Tellbound} finishes the proof.
\end{proof}

\begin{proof}[Proof of Theorem \ref{mom_num_evalues}]
Taking expectations in the inequality of Corollary \ref{num_evalue_transfer} and applying Fubini yields
\begin{align*}
\E ((N_n(E) -1)^+)^{3/2}
&\leq (n\abs{\Delta_n(E)})^{1/2} \int_{\Delta_n(E)}  \, \sum_{\ell = 1}^n \E \norm{M_n(x, \ell)}^3 \,dx.
\end{align*}
By the fact $\abs{\Delta_n(E)} = 2 R /(\rho(E) n)$, and Lemma \ref{norm_process_bound} the right hand side is bounded above by $
sup_{x \in \Delta_n(E)} f(x)$
for some continuous function $f$ on $(-2,2)$. Now fix $ \epsilon >0$ and $I_\epsilon =  (-2 + \epsilon, 2 - \epsilon)$.  There is an $N \in \N$ such that for any $n \geq N$ if $E \in I_{\epsilon}$, then $\Delta_n(E) \subset I_{\epsilon/2}$. Since $f$ is continuous on $(-2,2)$ this means that for $n \geq N$,
\[
\E ((N_n(E)-1)^+)^{3/2}  \leq  C_\epsilon. \qedhere
\]
\end{proof}

\section{Appendix}

\begin{theorem} \label{analytic_fact}
Let $D\left( [0,1], \C \right)$ be the space of cadlag functions from $[0,1]$ to $\C$. Suppose the sequence $f_n \in D\left( [0,1],\C \right)$ converges uniformly to $f \in C\left( [0,1], \C\right)$. Then for fixed $z \in \C$, $\abs{z} =1$ but $z \neq 1$,
\item
$$
\lim_{n \to \infty}\int_0^1 f_n(t)z^{\nfloor} dt = 0.
$$
\end{theorem}
\begin{proof}
Since
$$
\abs{ \int_0^1 f_n(t)z^{\nfloor} dt - \int_0^1  f(t) z^{\nfloor} dt } \leq \norm{f_n -f},
$$
it suffices to show that for any continuous $f: [0,1] \to \C$,
\[
\lim_{n \to \infty} \int_0^1 f(t)z^{\nfloor} dt =0.
\]
We first assume that $f$ is simple, by which we mean that
\(
f :=  c \mathbf{1}_{[a, b)},
\)
for some constant $c$ and subinterval $(a,b) \subset [0,1]$. We have that
\begin{align*}
\int_0^1 f(t) z^{\nfloor} & = \frac{c}{n} \sum_{k= \lceil n a \rceil}^{\lfloor nb \rfloor} z^k + o \left( \frac{1}{n} \right).
\end{align*}
Since $ z \neq 1$,
\(
\sum_{k= 0}^{N} z^k
\)
is bounded for all $N \in \N$, which finishes this case.  Additivity then gives the result for any finite sum of piecewise, simple functions. And for a general $f \in C\left([0,1], \C \right)$, we can find functions $g_m$ which are finite sums of simple functions so that
$$
\sup_n \abs{\int_0^1 g_m(t) z^{\nfloor}dt - \int_0^1 f(t)z^{\nfloor}dt } \leq \int_0^1  \abs{g_m(t) - f(t)} dt < \epsilon_m,
$$
with $\epsilon_m \to 0$.  This completes the proof.
\end{proof}

\begin{lemma} \label{rho_integral}
Let $\rho(x) = 1/ \sqrt{1 -(x/2)^2}$. Fix $\epsilon >0$ and $F \in C_c(\R)$. Then
$$
 \sup_{ \abs{\mu} < 2- \epsilon}
 \abs{ \int F\left(n\rho(x)(\mu-x) \right) \rho(x) dx -  \int F(x) dx  } = o\left( \frac{1}{n} \right)
$$
\end{lemma}
\begin{proof}
Suppose that $\supp F \subset [-R, R]$ for some $R>0$. Then we can suppose $\abs{\mu - x} \leq R/n$ because otherwise  since $\rho \geq 1$ we have that $F(n \rho(x) (\mu - x) ) =  F(n \rho(\mu) (\mu - x) ) = 0$. $\rho$ is Lipschitz on any closed subset of $(-2,2)$ and so for $n$ large enough (depending only on $\epsilon$) we have that
\begin{itemize}
\item $ \abs{\rho(\mu) - \rho(x)} \leq C/n$,
\item $ \abs{n \rho(\mu)(\mu - x) - n \rho(x) (\mu - x)}  \leq \frac{R C}{n}$.
\end{itemize}
This implies that
\begin{align*}
\int \abs{    F\left(n\rho(x)(\mu-x) \right) \rho(x) dx -   F\left(n\rho(x) (\mu-x) \right) \rho(\mu) dx  }
& \leq \frac{C}{n} \int F \left(n\rho(x)(\mu-x) \right) dx \\
& \leq \frac{C R \norm{F} }{n^2}.
\end{align*}
And also that,
\begin{align*}
\int & \abs{    F\left(n\rho(x)(\mu-x) \right) \rho(\mu) dx -   F\left(n\rho(\mu) (\mu-x) \right) \rho(\mu) dx  }
\\ & \leq \rho(\mu) \sup_{\abs{x-y} \leq C R /n} \abs{F(x) - F(y)} \int \mathbf{1}[\abs{\mu - x} < R/n] dx \\
& \leq \frac{D}{n} \sup_{\abs{x-y} \leq C R /n} \abs{F(x) - F(y)}\\
& = o\left(1/n \right)
\end{align*}
since $F$ is uniformly continuous. These two inequalities imply
$$
\sup_{ \abs{\mu} < 2- \epsilon}
 \abs{ \int F\left(n\rho(x)(\mu-x) \right) \rho(x) dx -  \int F\left(n\rho(\mu)(\mu-x) \right) \rho(\mu) dx   } = o\left(1/n \right).
$$
And we are done since $\int F\left(n\rho(\mu)(\mu-x) \right) \rho(\mu) dx = \int F(x) dx$.
\end{proof}

\begin{lemma}\label{rho_integral2}
Let $\rho(x) = 1/\sqrt{1 -(x/2)^2}$ and take $F \in C_c(\R)$ with $F\geq 0$ and $F(x) < F(y)$ for $\abs{x} > \abs{y}$. Then,
$$
\sup_{ \abs{\mu} < 2 }\int_{-2}^2 F\left( n \rho(x) (x - \mu) \right) \rho(x) dx \leq O\left( \frac{1}{n} \right) .
$$
\end{lemma}
\begin{proof}
By symmetry of $\rho(x)$, we can assume $\mu \geq 0$. Since $\rho(x) \geq 1$, we have that $\abs{x-\mu} \leq R/n$, where $\supp F \subset [-R,R]$. In particular, since $\mu \geq 0$, for $n$ large enough, we have that $x$ is bounded away from $-2$ independently of $\mu$. And so we can write
$$
\frac{c_1 }{\sqrt{2-x}} \leq \rho(x ) \leq \frac{c_2 }{\sqrt{2- x}}.
$$
The decreasing property of $F$  gives that
\begin{align*}
\int_{-2}^2 F\left( n \rho(x) (x - \mu) \right) \rho(x) dx
& \leq  c_2 \int_{-2}^2 F\left( c_1 n \frac{x - \mu}{\sqrt{2-x}} \right) \frac{dx}{\sqrt{2-x}}.
\end{align*}
Writing $\gamma = 2 - \mu$ and changing variables $y = \sqrt{2-x}/\sqrt{\gamma}$,
\begin{align*}
\int_{-2}^2 F\left( c_1 n \frac{x - \mu}{\sqrt{2-x}} \right) \frac{dx}{\sqrt{2-x}}.
&=\sqrt{\gamma} \int_{0}^{2/\sqrt{\gamma}}  F\left( c_1 n \sqrt{\gamma} \left( \frac{1-y^2}{y}\right) \right) d y\\
& \leq  C \norm{F} \sqrt{\gamma}  \int_{0}^{\infty} \mathbf{1}{\bigl [ \abs{y-1/y} \leq R/(n\sqrt{\gamma}) \bigr]} \, dy.
\end{align*}

\noindent Now fix $\alpha > 0 $. Notice that if $0 \leq x \leq 1$,
\begin{align*}
\abs{x - 1/x} \leq 2 \alpha \implies x  \geq \sqrt{\alpha^2 + 1} - \alpha .
\end{align*}
And so
\begin{align*}
\int_0^1 \mathbf{1}\left[ \abs{x - 1/x} \leq 2 \alpha\right]
& \leq 1 + \alpha - \sqrt{\alpha^2 +1} \\
& \leq C \alpha.
\end{align*}
Similarly if $x \geq 1$, then
\begin{align*}
\abs{x - 1/x} \leq 2 \alpha \implies x  \leq \alpha +  \sqrt{\alpha^2 + 1}.
\end{align*}
And so
\begin{align*}
\int_1^\infty \mathbf{1}\left[ \abs{x - 1/x} \leq 2 \alpha\right]
& \leq \alpha -1 + \sqrt{\alpha^2 +1} \\
& \leq C \alpha.
\end{align*}
Therefore
\begin{align*}
\sqrt{\gamma}  \int_{0}^{\infty} \mathbf{1}{\bigl [ \abs{x-1/x} \leq R/(n\sqrt{\gamma}) \bigr]} \, dx
& \leq C \sqrt{\gamma} \frac{R}{n \sqrt{\gamma}} \\
& = C/n.
\end{align*}
\end{proof}

We finish with the standard result showing that the expected eigenvalue distributions converge.
\begin{lemma} \label{l:dos}
Let $H_n$ be as in \eqref{shrod1dmatrix} and let $\nu_n=\tfrac{1}{n}\sum_{k=1}^n\delta_{\mu_k}$ be its empirical eigenvalue distribution.
Then for every interval  $A$ we have $\ev \nu_n(A) \to \nu_\infty(A)=\int_A \tfrac{\rho(s)}{2\pi}\,ds$, the arcsine law with $\rho$ as in \eqref{defrho}.
\end{lemma}
\begin{proof}
This is standard, but we have not found it in the literature for this precise setting. First assume that the random variables $\omega_i$ are bounded.
Then the standard-path counting shows that the moments of $\ev \nu_n$ converge to those of $\nu_\infty$, and so the measures converge weakly.
Otherwise, let $\nu_n^b$ be the measures $\nu_n$ based on the random variables $\omega_i$ truncated at $b$. Since the rank of a perturbation is an upper bound
on the Kolmogorov-Smirnoff distance $d_{KS}$ of the eigenvalue counting measures, we have that
$$
\ev d_{KS}(\nu_n,\nu_n^b)\le P(|\omega_i|>b).
$$
The convexity of the distance function and Jensen's inequality applied twice gives
$$
d_{KS}(\ev\nu_n,\ev\nu_n^b)\le P(|\omega_1|>b).
$$
Since for all $b$ we have $\ev \nu_n^b\to \nu_\infty$ weakly, the standard diagonalization now completes the proof.
\end{proof}
\bigskip

\noindent{\bf Acknowledgements.} We thank both Mustazee Rahman and an anonymous referee for helpful comments on a previous version. The second author was supported by the Canada Research Chair program, the NSERC Discovery Accelerator
grant, the MTA Momentum Random Spectra research group, and the ERC
consolidator grant 648017 (Abert).

\bibliographystyle{dcu}
\bibliography{RSOBib}

\bigskip\bigskip\bigskip\noindent
\begin{minipage}{0.49\linewidth}
Ben Rifkind
\\Department of Mathematics
\\University of Toronto
\\Toronto ON~~M5S 2E4, Canada
\\{\tt ben.rifkind@gmail.com}
\end{minipage}
\begin{minipage}{0.49\linewidth}
B\'alint Vir\'ag
\\Departments of Mathematics and Statistics
\\University of Toronto
\\Toronto ON~~M5S 2E4, Canada
\\{\tt balint@math.toronto.edu}
\end{minipage}
\end{document}